\documentclass[12pt]{article}

\usepackage{epic}
\usepackage{amsmath}[1996/11/01]
\usepackage{url,amssymb,amsthm,amsfonts,latexsym,epsfig,graphics}
 \def\beql#1#2\eeql{\begin{equation}\label{#1}#2\end{equation}}

\hsize = \textwidth
\DeclareMathOperator{\ind}{ind}
\DeclareMathOperator{\inv}{inv}
\DeclareMathOperator{\Inv}{Inv}

\DeclareMathOperator{\Br}{Br}
\DeclareMathOperator{\Gal}{Gal}
\DeclareMathOperator{\sol}{sol}

\DeclareMathOperator{\Irr}{Irr}

\DeclareMathOperator{\GL}{GL}

\DeclareMathOperator{\diag}{diag}

\theoremstyle{plain}
\newtheorem{theorem}{Theorem}
\newtheorem{lemma}[theorem]{Lemma}

\newtheorem{assumption}[theorem]{Assumption}

\newtheorem{corollary}[theorem]{Corollary}
\newtheorem{definition}[theorem]{Definition}
\theoremstyle{remark}

\newtheorem{remark}[theorem]{Remark}

\numberwithin{theorem}{section}

\newcommand{\PP}{{\mathbb{P}}}
\newcommand{\Z}{{\mathbb{Z}}}
\newcommand{\Q}{{\mathbb{Q}}}

\newcommand{\N}{{\mathbb{N}}}
\newcommand{\R}{{\mathbb{R}}}

\newcommand{\C}{{\mathbb{C}}}
\newcommand{\trace}{\mbox{trace}}

\title{The solvable degree of a Schur algebra}
\author{Gabriele Nebe\footnote{gabriele.nebe@rwth-aachen.de}\  and David Schlang\footnote{david.schlang@rwth-aachen.de}}
\date{Lehrstuhl f\"ur Algebra und Zahlentheorie,  \\
RWTH Aachen University, Germany}

\begin{document}
\maketitle

 {\sc Abstract.} 
 The solvable degree 
 of a Schur algebra $D$ is the greatest common divisor of 
 all integers $n$ such that $D^{n\times n}$ is a direct summand of the rational
 group algebra $\Q G$ for some finite solvable group $G$. 
 The main result of this paper gives divisibility conditions on the solvable 
 degree depending on the center and the rational primes that divide the
 discriminant of $D$. 
  {\sc keywords:}  Schur subgroup of Brauer group; direct summands of group algebras; 
  Schur index; character degree

\section{Introduction} 

An interesting number theoretic invariant of a 
finite group $G$ is its rational group ring $\Q G $. 
This is a $|G|$-dimensional semi-simple $\Q $-algebra, 
 $$\Q G = \bigoplus _{i=1}^s D_i^{n_i \times n_i }$$ 
 for rational division algebras $D_i$. 
Their centers, $K_i := Z(D_i)$, 
are exactly the character fields 
 $\Q(\chi _i)$ for the (Galois orbits of) irreducible complex characters $\chi _i$. 
 In particular $K_i = \Q[\chi _i(g) : g\in G]$ is generated by sums of $|G|$-th 
 roots of unity and hence an abelian extension of $\Q $.
 Following Turull \cite{Turull} we call the class $[D_i]:= [ \chi _i ] $ the (inverse) 
 Brauer element of $\chi _i$ in the Brauer group  $\Br(K_i)$ of central simple $K_i$-algebras. 


The Schur subgroup $S(K)$ of the Brauer group $\Br(K)$
is the group of classes $[D]$ of all central simple $K$-algebras 
for which there is some $n\in \N$ and some finite group $G$ such that 
$D^{n\times n}$ is a direct summand of the rational group algebra $\Q G$. 
 The group $S(K)$ has been extensively studied since the middle of last century 
(see \cite{Yamada} and the more than 70 papers citing this book). 
 We call a  division algebra $D$ a {\em Schur algebra} 
 if  $[D] \in S(Z(D))$.
 The {\em degree} $\deg(D)$ (respectively the {\em solvable degree} $\deg_{\sol}(D) $) 
 of a Schur algebra $D$ is the 
 greatest common divisor of the values $n$, such that 
 $D^{n\times n}$ is a direct summand of $\Q G$ for some finite group  
 $G$ (respectively some finite solvable group $G$). 
 Clearly $\deg(D) $ divides $\deg_{\sol}(D)$. There are Schur algebras 
 $D$, where $\deg(D) \neq \deg_{\sol}(D)$ (see Remark \ref{degeq1}). 

 For solvable groups, there are 
divisibility relations for the character degrees depending on 
the character field (see Theorem \ref{Cram}). 
 These provide our main tool to show the main result of this 
 paper, Theorem \ref{main}, giving divisibility conditions on the 
 solvable degree of a Schur algebra. 
 We also give examples that prove that these are tight for certain 
 infinite series of Schur algebras. 

 Section 2 of this paper summarizes the relevant results from the literature that 
 are used in the proof of Theorem \ref{main}. 
 Section 3 might be of independent interest; here we generalize the results 
 from \cite{Yamadamon} to obtain conditions on the Schur index of an induced 
 character. A more sophisticated theory based on Tignol's notion of generalised crossed 
 product algebras is developed in \cite{Schlang}. 
 Section 4 then introduces the degree and solvable degree of a Schur algebra 
 and the last section is devoted to the statement and proof of our main result. 

The paper is a contribution to Project-ID 286237555 -- TRR 195 -- by the
Deutsche Forschungsgemeinschaft (DFG, German Research Foundation).

 \section{Characters} 

 \subsection{The Brauer element} 

	The set of rational primes is denoted by 
	$\PP :=\{ p \in \N \mid p \mbox{ is a prime} \} $.
\\
Let $K$ be a number field and $D$ be a division algebra with center $K$. 
Then $\dim_K(D) = m^2$ where $m:=\ind(D)$ is the index of $D$. 
Note that the index, $m$, is also the order of $[D]$ in the Brauer group $\Br(K)$.
Division algebras over number fields satisfy the Hasse principle, 
i.e.. they are isomorphic, if and only if all their completions are 
isomorphic: 
\\
 A place $P$ of $K$ is either a prime ideal of the ring of
      integers in $K$ (a finite place) or an embedding into the
        real or complex numbers (an infinite place).
The Brauer group of the corresponding completion 
$\Br (K_P)$ is either isomorphic to $\Q/\Z $ (for the finite places),
to $\Br(\R) \cong \frac{1}{2}\Z/\Z$, or to $\Br(\C) = 1$.
        The completion 
        $A_P = D\otimes_K K_P$ of $D$  is a matrix ring over some $K_P$-division
        algebra $A_P = D_P^{n_P\times n_P} .$
	The division algebra $D_P$ is uniquely determined by its 
	{\em Hasse invariant} 
	$$\inv _P(D) = \inv(D_P) = \frac{a_P}{m_P} + \Z \in \Br (K_P )$$
	where $m_P = m/n_P$ is the index of $D_P$.

 For a finite group $G$ we denote by 
 $$\Irr (G) := \{ \chi : G \to \C \mid \chi \mbox{ is an irreducible character } \} $$
 the set of irreducible complex characters of $G$. 
 The characters $\chi \in \Irr(G) $ correspond to the conjugacy classes of 
 irreducible complex representations 
 $$\rho : G \to \GL_{\chi (1)} (\C ), \chi (g) = \trace (\rho  (g)) \mbox{ for all } g\in G .$$
 We then also say that $\rho $ affords the character $\chi $.
 The character $\chi $ is called {\em faithful} if $\rho $ is 
 injective. 
The $\Q $-algebra spanned by $\rho (G)$ is a simple algebra, 
$$\langle \rho (g) \mid g\in G \rangle _{\Q } \cong D^{n\times n} \leq \C ^{\chi(1) \times \chi(1) } $$ 
where $K := Z(D) = \Q(\chi  )$ is the character field of $\chi$.

 \begin{definition} 
			 Let $G$ be a finite group,
			 $\chi \in \Irr(G)$,  $K=\Q(\chi )$, $\rho $ and $D$ as above.
	 \begin{itemize} 
		 \item[(a)] (see \cite{Turull}) 
			 The {\em inverse Brauer element} of $\chi $ is the class
	 $$[\chi ] = [D ]\in \Br (K) $$  
	 in the Brauer group of $K$.
 \item[(b)]  
	 The {\em Schur index} $m(\chi)$
			 is the order  
			 $m(\chi ) = |\langle [\chi ] \rangle |$ 
			 of the Brauer element of $\chi $.
		 \item[(c)] For two fields 
 $K\leq L$  we denote the natural morphism
	$\Br(K ) \to \Br(L), [D] \mapsto [D\otimes L] =: [D]\otimes L$.
 \item[(d)]
For a place $P$ of $K$ we define 
			 the $P$-local Schur index 
			 $m_P(\chi )$ to be the order of $[\chi ]\otimes K_P \in \Br(K_P)$
			 in the Brauer group of the $P$-adic completion of $K$.
		 \item[(e)] 
			 The division algebras $D$ with $[D] = [\chi] $ 
			 for some finite group $G$ and 
	 some $\chi \in \Irr(G)$ are called {\em Schur algebras}. 
	 The Schur subgroup $S(K)$  of the Brauer group $\Br(K)$ is the 
	 subgroup consisting of all classes of Schur algebras with center $K$. 
	 \end{itemize}
 \end{definition}

 \begin{remark} \label{mdivdeg}
	 As the index of $[\chi ]$ always divides the character 
	 degree $\chi (1)$ we state the 
	 well known fact that $m(\chi ) $ divides $\chi (1)$.
 \end{remark}

 Note that Turull defines the Brauer element of $\chi $ as 
 the inverse of our inverse Brauer element. 
 As the natural involution on $\Q G$, inverting the group elements, 
 maps $[\chi ] $ to $[\chi ]^{-1} = [\overline{\chi } ]$, the 
 Brauer element of $\chi $ is just the inverse Brauer element of 
 the complex conjugate character $\overline{\chi }$.
 So both have the same degree and also the same solvable degree.

Algorithms for computing Schur indices of characters are developed in 
\cite{Lorenz} and \cite{Unger}. They are based on the restriction 
to so called elementary subgroups and the following result from Brauer: 

\begin{theorem} (see \cite[Section (2A)]{Brauer}, \cite[Lemma 3.2]{Unger}) \label{brauer}
	Let $U \leq G$ be finite groups, $\chi \in \Irr (G)$ 
	and $\psi \in \Irr (U)$, 
	occurring in the restriction 
	$\chi _{|U}$ with multiplicity $e:= (\chi _{|U}, \psi)_U$. 
	If $p \in \PP $ does not divide 
	$e [\Q(\psi , \chi  ): \Q(\chi ) ]$ then the $p$-part of 
	the Schur index of $\psi $ over $\Q(\chi )$ 
	equals to the $p$-part of $m(\chi )$. 
\end{theorem}

Theorem \ref{brauer} also allows for computations of the inverse Brauer element, 
where we only need one special case:

\begin{corollary} \label{brauer1}
	 Let $U \leq G$ be finite groups, $\chi \in \Irr (G)$
        and $\psi \in \Irr (U)$. 
	If $\psi $ occurs in $\chi _{|U}$ with multiplicity 1 and 
	$\Q ( \psi ) \leq \Q(\chi )$ then $[\chi ] = [\psi ]\otimes \Q(\chi )$.
\end{corollary} 

\begin{proof}
	Put $K:=\Q(\chi )$. 
	Then 
	the order of $[\chi ]$ equals to the order of $[\psi ]\otimes K$
	in the Brauer group of $K$ by Theorem \ref{brauer}.

Let $\rho $ and $\delta $  be complex representations affording the character $\chi $ 
	respectively $\psi $  and put
	$$A := \langle \rho (g) \mid g\in G \rangle _{K} \mbox{ and } 
	B := \langle \delta  (u) \mid u\in U \rangle _{K} .$$
	Then $A\cong D^{n\times n}$ for some $n$, 
	where  $D$ is the central $K$-division algebra 
	with $[D] = [\chi ]$.
	Put 
	$$e_{\psi }:= \frac{1}{|U|} \sum_{u\in U} \psi(u^{-1}) \rho (u)  .$$
	Then $e_{\psi }$ is an idempotent in $A$ and  
	 $e_{\psi } A e_{\psi } \cong B $ because 
	 the multiplicity of $\psi $ in $\chi _{|U }$ is 1.
	By \cite[Remark 4.10]{UnitaryDisc} we have 
	$e_{\psi } A e_{\psi }  \cong D^{k\times k}$ for some $k\leq n$, so 
	$[D] = [B] = [\psi ] \otimes K $.
\end{proof}

 The two papers \cite{B} and \cite{BS} list important properties of Schur algebras:

\begin{theorem}\label{unidist} 
	Let $D$ be a Schur algebra, $K:=Z(D)$ its center and $m=\ind(D)$ the 
	index of $D$.
	\begin{itemize} 
		\item[(i)] 
$K$  is an abelian number field
containing primitive $m$-th roots of unity (\cite[Theorem]{BS}), so
	$$\zeta _m := \exp(\frac{2\pi i}{m} ) \in K .$$
\item[(ii)] 
	Let $p\in \PP$. Then there is a divisor $m_p$ of $m$ such 
			that $m_P=m_p$ for all 
 finite places $P$ of $K$ over $p$ (\cite[Theorem 1]{B}). 
 \item[(iii)] If $2<p\in \PP $ then $m_p$ of (ii) divides $p-1$ 
	 and $m_2 \in \{1,2 \}$ (\cite[Satz 10, Satz 11]{Witt}).
\item[(iv)]
	 Let $P$ be a prime ideal of $K$
	 and $\sigma \in \Gal (K/\Q )$, $b\in \{1,\ldots , m-1\}$ be such 
	 that $\sigma (\zeta _m) = \zeta _m^b$.
	 Then the Hasse invariants of $D$ at $P$ and $\sigma (P)$ satisfy 
	 $\inv _P(D) = b \inv_{\sigma(P)}(D) .$
\end{itemize} 
\end{theorem}

Since the Galois group acts transitively on the prime ideals over $p$, 
Theorem \ref{unidist} (iv) implies that the $p$-local invariants are
fully determined by the Hasse invariant for one prime ideal over $p$. 
Conway polynomials provide a tool to single out a unique prime ideal 
$P:=P(K,p)$ of $K$ that contains $p$ (see \cite[Introduction, 3)]{Braueratlas}). 
Then any Schur algebra $D$ with $Z(D) = K$ is 
uniquely determined by the finite set 
$$\Inv(D) := \{ (\inv _{P(K,p)} ,p) \in \Q/\Z \times \PP  \mid p\in \PP , m_p(D) > 1 \} \cup  I $$ 
where $I=\emptyset $ if no real place of $K$ ramifies in $D$ and 
$I=\{ \infty \} $ otherwise.
We also put $\Inv ([\chi ]) = \Inv(D)$ for irreducible characters $\chi $ with 
$[\chi ] = [D]$.

For many purposes the following coarser invariants suffice:

\begin{definition}\label{Pchi}
	For a finite group $G$, 
	$\chi \in \Irr(G)$, $p\in \PP $,
	and $D$ a division algebra such that $[\chi ] = [D] $ we put
	$$m_p(D) = m_p([D]) = m_p ([\chi ]) = m_p (\chi ) $$ 
	to denote the $p$-local Schur index of $\chi $.
	The set of ramified primes in $[\chi ] $ is denoted by
$${\mathcal P}([\chi ]) := {\mathcal P}(\chi ) := \{ p \in \PP \cup \{ \infty \} \mid m_p(\chi ) > 1 \}$$  and 
	$ P(D):= P(\chi ):=  \prod  \{ p \in {\mathcal P}(\chi ) \mid p\neq \infty \} 
	=\prod \{ p \in \PP \mid m_p(\chi ) > 1 \}.$
\end{definition}

\subsection{Clifford Theory} 

 Clifford Theory deals with the restriction and induction 
 of characters for normal subgroups. 
 Let $G$ be a finite group and $N\unlhd G$ be a normal subgroup of $G$. 
 Then $G$ acts on $\Irr (N)$ by conjugation: 
 For $\psi \in \Irr(N)$ and $g\in G$ put $\psi ^g : N \to \C $, 
 $\psi ^g(n) = \psi (gng^{-1}) $ for all $n\in N$.

 \begin{theorem} (Clifford, see for instance \cite[Theorem (6.5)]{Isaacs})
The restriction of  $\chi \in \Irr (G)$ to the normal subgroup $N$ is of the form
	 $$\chi _{|N} = e (\psi _1 + \ldots + \psi _t ) $$ 
	 where $\psi _i \in \Irr (N)$ and $e\in \N $. 
Moreover $\{ \psi_1,\ldots , \psi_t \} $ is a full orbit 
under $G$. 
Put $\psi := \psi _1$. 
	 Then the stabiliser 
	 $$T:= I_G(\psi) :=  \{ g\in G \mid \psi ^g = \psi  \} $$ 
	 is a subgroup of $G$ of index $t = [G:T]$. 
	 There is a unique $\alpha \in \Irr (T) $ such that 
	 $$\alpha _{|N } = e \psi  \mbox{ and } \chi = \alpha ^G .$$
	 The character $\alpha $ is called the 
	 {\em Clifford correspondent} of $\chi $ over $\psi $.
	 \\
	 For the character fields we get 
	 $$ \Q(\psi) \subseteq  \Q(\alpha )  \mbox{ and } 
	  \Q(\chi) \subseteq  \Q(\alpha ) . $$
 \end{theorem}

 \begin{theorem} \label{Gallagher} (Gallagher, see for instance \cite[Corollary (6.17)]{Isaacs} )
	 If $\chi \in \Irr(G)$ is such that $\chi _{|N} =: \psi \in \Irr(N)$, then 
	 the set of irreducible constituents of the induced character 
	 $\psi ^G$ is 
	 $\{ \beta \chi \mid \beta \in \Irr(G/N) \} $.
 \end{theorem} 

\subsection{$p'$-special characters of solvable groups} 

In the proof of the main result we use  $p'$-special characters. 
For a finite group $G$ and a character $\chi \in \Irr(G)$ 
we define the determinantal order $o(\chi )$ to be 
$$| \langle \det(\rho (g) ) \mid g\in G \rangle | $$ 
the order of the subgroup of the multiplicative group 
$\Q (\chi )^{\times }$ generated by 
the determinants of the images of the group elements under any complex  
representation $\rho $ affording the character $\chi $. 

\begin{definition}
 Let $p$ be a prime and let $G$ be a solvable group.
	An irreducible character $\chi \in \Irr(G)$ is called {\em $p'$-special} if
	 \begin{itemize}
		 \item[(i)] $p$ does not divide $\chi(1)$, and
		 \item[(ii)] for every subnormal subgroup $S $ of $G$ and every irreducible constituent $\theta \in \Irr(S)$ of $\chi_{\mid S}$, the determinantal order $o(\theta)$ is prime to $p$.
	 \end{itemize}
\end{definition}

From {\cite[Lemma 3.3]{Isaacspaper} we obtain 
 \begin{lemma}\label{pspeciallemma}
 Let $G$ be solvable, $N \unlhd G$, and let $\theta \in \Irr(N)$ be $p'$-special.
	 \begin{itemize}
		 \item[(i)] If $G/N$ is a $p'$-group, then every irreducible constituent of $\theta^G$ is $p'$-special.
		 \item[(ii)] If $G/N$ is a $p$-group and $\theta$ is $G$-invariant, then $\theta$ has a unique extension $\hat{\theta} \in \Irr(G)$ whose determinantal order $o(\hat{\theta})$ is prime to $p$. Moreover this extension $\hat{\theta}$ is $p'$-special.
	 \end{itemize}
 \end{lemma}

 We use the following property of $p'$-special characters: 

\begin{lemma} \label{SchurIndex1} ({\cite[Proposition 6.3 (b)]{GAJENDRAGADKAR1979237}})
	Let $\chi \in \Irr (G)$ be $p'$-special and let $\chi' $ denote the 
restriction of $\chi $ to the $p$-regular elements. 
	Then $\chi' $ is an irreducible Brauer character of $G$. 
\end{lemma} 

\begin{remark} \label{remSchurIndex1}
	Let $\varphi $ be an irreducible Brauer character of $G$, $\chi \in \Irr(G)$ 
	and $d_{\chi ,\varphi }$ denote the decomposition number, i.e. the
	multiplicity of $\varphi $ in the restriction of $\chi $ to 
	the $p$-regular elements of $G$. 
	Then 
                        \[
				m_p(\chi) \bigm| d_{\chi ,\varphi} [\Q_p(\chi, \varphi ) : \Q_p(\chi)] 
                        \]
			(see \cite[Theorem 4.5(1)]{Unger}). 
			In particular 
			Lemma \ref{SchurIndex1} implies that 
			for any $p'$-special character  $\chi $
		 the $p$-local Schur index $m_p(\chi )=1$.
\end{remark}

\subsection{Character degrees for solvable groups}

 For an abelian number field $K$, 
 the  {\em conductor} $f(K)$ is the 
 smallest natural number $n$ such that $K$ is contained in the $n$-th cyclotomic field,
 $K\subseteq \Q[\zeta _{n}]$.
 We also put $F(K) := \Q[\zeta _{f(K)}] $ to denote the conductor field.

 \begin{theorem} \cite{Cram} \label{Cram} 
	 Let $\chi $ be an irreducible character of some finite solvable group $G$. 
	 Then $[F(\Q(\chi)) : \Q(\chi )] $ divides $\chi (1) $.
 \end{theorem}

\section{Schur indices of induced representations} 

This section gives a generalization of \cite[Theorem 2]{Yamadamon}. 
It is useful to apply the criterion prime per prime so we will work over 
local fields and consider the following situation: 

$G$ is a finite group, $N\unlhd G$ a normal subgroup, $\alpha \in \Irr (N)$ such 
that $\alpha ^G = \chi \in \Irr (G) $. 
Then $\chi (1) = [G:N] \alpha(1)$ and the restriction  is 
$$\chi _{|N} = \sum _{gN \in G/N} \alpha ^g .$$

Let $p$ be a rational prime.
Then the $p$-adic character fields, $E:=\Q_p(\alpha )$ and $K:=\Q_p(\chi )$,
satisfy $K\leq E$. 

\begin{theorem} \label{Yamadagen}  
	If $E/K$ is unramified and $m_p(\alpha ) = 1$ then $m_p(\chi ) $ divides 
	$\gcd (\alpha(1),[E:K]) $.
\end{theorem} 

\begin{proof}
	(a) 
Similarly as in \cite{Yamadamon} we first construct a subgroup 
	$H\leq G$ such that $H/N \cong \Gal(E/K)$: 
	$$H:=\{ g\in G \mid \exists \sigma \in \Gal(E/K ) : \alpha (gng^{-1}) = \sigma(\alpha (n) ) \mbox{ for all } n\in N \} .$$
	(b)
	For the induced character $\psi = \alpha ^H $ 
	we then get $\Q_p(\psi ) = \Q_p(\chi) $ and $\chi = \psi ^G$ and hence 
	$m_p(\psi ) = m_p(\chi )$ by Corollary \ref{brauer1}.
	\\
	(c) 
	The fact that 
	$m_p(\alpha ) = 1$ implies that $E$ splits $[\psi ]$ and 
	$m_p(\psi )$ divides $[E:K]$. 
	\\
	(d) So it is enough to show that $m_p(\psi ) \mbox{ divides } \alpha(1)$: 
	\\
	Put $d:=\alpha (1)$ and $t:=[E:K] = |H/N |$.
	\\
	If $t=1$, then $H=N$, $\psi = \alpha $ and (d) follows from Remark \ref{mdivdeg}.
	So assume that $t\geq 2$.\\
	Since $E/K$ is unramified, its Galois group is cyclic, 
	$\Gal (E/K) = \langle \sigma \rangle .$ \\
	Choose $h\in H$ with $\alpha (hnh^{-1}) = \sigma(\alpha (n) ) \mbox{ for all } n\in N$ and let $\rho : N \to \GL_{d}(E)$ be a representation affording the character $\alpha $. 
	By abuse of notation, we apply elements of $\Gal(E/K)$ also to
	matrices over $E$, by applying the Galois automorphisms to all entries 
	of the matrix. 
	The representations $\rho ^{\sigma } $ and $\rho ^h $ with 
	$$\rho^{\sigma }(n) = \sigma (\rho (n)) \mbox{ and } 
	\rho ^h (n ) = \rho(hnh^{-1}) \mbox{ for all } n\in N$$
	 have the same character and hence are equivalent. 
	So there is a matrix $T\in \GL_d(E)$ such that 
	$$  \sigma (\rho (n) )  = T \rho (h n h^{-1}) T^{-1} \mbox{ for all } n\in N .$$
 For $0 \leq j \leq t$ put
                    $X_j := \sigma^j(T) \cdots \sigma(T)T $. 
		    Inductively we get
		    $$
                    \sigma^j(\rho(n)) = X_{j-1} \rho(h^j n h^{-j}) X_{j-1}^{-1} \text{ for all } n \in N, 1 \leq j \leq t.
		    $$
The induced representation $\rho^H: H \to \GL_{dt}(E)$ can be  defined by putting
                \[
			\rho^H(n) := \diag (\rho(n), \sigma(\rho(n)), \ldots, \sigma^{t-1}(\rho(n))) \text{ for all } n \in N
                \]
                and
                \[
                    \tilde{h} := \rho^H(h) = \begin{pmatrix}0 & T^{-1} & 0 & \cdots & 0 \\ 0 & 0 & \sigma(T)^{-1} & \ddots & \vdots \\ \vdots & \ddots & \ddots & \ddots & \vdots \\ 0 & \cdots & \cdots & 0 & \sigma^{t-2}(T)^{-1} \\ X_{t-2}\rho(h^t) & 0 & \cdots & \cdots & 0 \end{pmatrix}.
                \]
                As both $\alpha$ and $\psi$ are irreducible characters, the $K$-algebra $B:= \langle \rho^H(g) \mid g \in H \rangle_K$ is a central simple algebra containing the subalgebra $A:= \langle \rho^H(n) \mid n \in N \rangle_K$.
                Since $m_p(\alpha) = 1$, we have $A \cong E^{d \times d}$, by our choice of basis 
                \[
			A = \{\diag (M, \sigma(M), \ldots, \sigma^{t-1}(M)) \mid M \in E^{d \times d}\}.
                \]
                Conjugation with the matrix
                \[
			\tilde{\sigma} := \diag (T, \sigma(T), \ldots, \sigma^{t-1}(T)) \tilde{h} = 
		\left( \begin{array}{ccccc} 
		0 & I_d & 0 & \cdots & 0 \\ 0 & 0 & I_d & \ddots & \vdots \\ \vdots & \ddots & \ddots & \ddots & \vdots \\ 0 & \cdots & \cdots & 0 & I_d \\ X_{t-1}\rho(h^t) & 0 & \cdots & \cdots & 0 \end{array} \right) \in B
                \]
                induces the Galois automorphism $\sigma$ on $A$.
	In particular $\tilde{\sigma }^t \in A$ commutes with all elements in $A$,
	so is a scalar matrix 
	$$ \tilde{\sigma }^t = X _{t-1} \rho(h^t) \otimes I_t = \lambda I_{dt} $$ 
	for some $\lambda \in E^{\times }$. Then $\lambda \in K^{\times }$ because
	$\sigma(\lambda ) = \lambda $.
	As  $A\cong E^{d \times d} $ this establishes an 
	isomorphism between $B$ and the matrix ring 
	$(E/K,\sigma , \lambda ) ^{d\times d}$ 
	of the cyclic algebra $(E/K,\sigma , \lambda )$ (see \cite[Section 30]{Reiner}). 
	In particular 
	$$[B] = [(E/K,\sigma , \lambda )] \in \Br(K) $$ and hence 
	$$[B]^d = [(E/K,\sigma , \lambda^d )] \in \Br(K) .$$
	Now 
	$$\lambda ^d = \det (\lambda I_d) = \det (X _{t-1} \rho(h^t)  ) = N_{E/K}(\det(T)) \det(\rho(h^t)) .$$
	Let ${\mathcal O}_E$ resp. ${\mathcal O}_K$ denote the valuation rings of $E$ resp. $K$. 
As $E/K$ is unramified the norm is surjective on units, 
$N_{E/K}({\mathcal O}_E^{\times }) = {\mathcal O}_K^{\times }$. The determinant of $\rho(h^t)$ is an 
element of finite order in $K^{\times }$, so $\det(\rho(h^{t})) \in {\mathcal O}_K^{\times }$,
which shows that $\lambda ^d \in N_{E/K} (E^{\times })$. 
Therefore $[B]^d = [(E/K,\sigma , \lambda^d )] = [K] \in \Br(K)$, so the 
order $m_p(\psi )$ of $[B]$ in the Brauer group of $K$ divides $d = \alpha (1)$.
\end{proof}

\section{The degree of a Schur algebra}

Let ${\mathcal G}$ respectively ${\mathcal G}_{\sol}$ denote the category of all 
finite groups, respectively the one of all finite solvable groups.
Let $K$ be a number field.

 \begin{definition} 
	 For a central simple $K$-division algebra $D$ we put 
	 $$\deg(D) := n \mbox{ respectively } \deg_{\sol}(D) := n$$ if 
	 $n\Z $ is the ideal generated by 
	 $$\{ t \mid \exists G\in {\mathcal G} 
	 \mbox{ (resp. } {\mathcal G}_{\sol} \mbox{) such that } 
	 D^{t\times t} \mbox{ is a direct summand of } \Q G \} . $$
	 If $[D] \not\in S(K)$ then we put $\deg(D) = \deg_{\sol}(D) = \infty $.
 \end{definition} 

 \begin{remark}
	 Clearly the Galois group $\Gal(K/\Q) $ of $K$ acts on 
	 the Schur subgroup $S(K)$. Schur algebras in the same Galois 
	 orbit have the same degree. 
 \end{remark}

 \subsection{Some definite quaternion algebras of degree 1} 
 
 If the center of a Schur algebra $D$ is a totally real field, $Z(D) = K$, then
either $D=K$ or $D$ is a quaternion algebra.
In the latter case, there are
$a,b \in K$ such that
$$D = (a,b) _K = \langle 1,i,j,ij \mid i^2=a, j^2=b, ij=-ji \rangle .$$
Then $D$ is called a {\em definite} quaternion algebra, if $\infty  \in {\Inv}(D) $, i.e. $D\otimes _K \R $ is a division algebra
(for one and hence all real embeddings of $K$ into $\R $).

 The definite quaternion algebras $D$ 
 that are direct summands of $\Q G$ for some finite group $G$
 are classified in \cite[Theorem 6.1]{quat} (see also \cite{Amitsur}):

 \begin{theorem} \label{dim1}
         Assume that $K$ is a real number field and let $D$ be a
         definite quaternion algebra with center $K$ such that $D$ is
         a direct summand of $\Q G$ for some finite group $G$.
         Then one of the following holds:
         \begin{itemize}
                 \item
         For $K=\Q $ there are two such quaternion division algebras:
			 $D = (-1,-1)_{\Q}  $ (with $\Inv(D) = \{ (\frac{1}{2} ,2), \infty \}$)
			 and
			 $D = (-1,-3)_{\Q}$ (with $\Inv(D) = \{ (\frac{1}{2} ,3),\infty \} $).
 \item
         For $[K:\Q ] =2 $ there are three such quaternion division algebras,
			 the ones where $K= \Q[\sqrt{p}]$ for $p=2,3,5$ and $\Inv(D) = \{\infty \}$.
 \item
         If $[K:\Q ] \geq 3$ then $F(K) =\Q[\zeta _{m}]$ for some even natural number 
			 $m$,
                 $K$ is the maximal real subfield of $F(K)$, i.e. $[F(K):K] = 2$,
                 and $G=C_m.C_2 = Q_{2m} $ is the quaternion group of
                 order $2m$.  
			 $$\Inv(D) = 
			 \left\{ \begin{array}{ll} 
				 \{ \infty \}  & \mbox{ if } [K:\Q ] \mbox{ even } \\
				 \{ (\frac{1}{2} ,p), \infty \} & \mbox{ if } m = 2p^a \mbox{ for a prime } p\equiv 3 \pmod{4}.
			 \end{array} \right.  $$
         \end{itemize}
 \end{theorem}

There are more examples of quaternion algebras $D$
 with $\deg(D) = 1$: 

 \begin{remark}\label{degeq1}
 Take $D=(-1,-11)_{\Q }$ (i.e. $Z(D) = \Q $, $\Inv(D) = \{ (\frac{1}{2},11),\infty \}$).
  Then $D^{5\times 5}$ is a direct summand of $\Q G$ for
    $$G=C_{22} . C_{10}  = \langle a,b \mid a^{22}, b^{10} = a^{11}, (a^2)^b = a^4 \rangle .$$
       By \cite{Feit} (see also \cite{Unger} for some corrections)
       the group $H=2.M_{22}.2i$ has a character $\chi $ of degree $252$
       with $[\chi ] = [D] $, so $\Q H$ has a direct summand
       $D^{126 \times 126}$.
       So here we find $\deg(D) = 1$.
 However, Theorem \ref{main} below shows that $\deg_{\sol}(D) = 5$.
 \end{remark}

 \subsection{Schur algebras of even degree} 

 Using the classification of finite simple groups, Navarro and Tiep have 
 shown the following result on character fields of odd degree irreducible 
 characters:

\begin{theorem}(\cite[Theorem A1]{NavarroTiep}) \label{odddeg}
Let $\psi $ be an absolutely irreducible
ordinary character of some finite group $G$
such that $\psi(1) $ is odd.
If $\zeta _{2^a} \in F(\Q(\psi ))$ then $\zeta _{2^a} \in \Q(\psi )$.
\end{theorem}

This theorem implies that there are Schur algebras $D$ with 
$\deg(D) \neq 1$: 

\begin{corollary} \label{odddegD}
	Let $D$ be a Schur algebra of odd index such that 
	$K:=Z(D)$ does not satisfy Theorem \ref{odddeg}, i.e. there 
	is $a\in \N$ such that $\zeta _{2^a} \in F(K) \setminus K$. 
	Then $\deg(D) $ is even.
\end{corollary} 

It is easy to construct finite groups $G$ and 
irreducible characters $\chi \in \Irr (G)$ such that $[\chi ] = [D]$ as in 
Corollary \ref{odddegD}. 
Put $H:=C_7 : C_9$ to be one of the groups from Corollary \ref{cor2} below. 
Then there is a faithful character $\chi \in \Irr (H)$ with $\Q(\chi ) = \Q[\sqrt{-7},\zeta_3] =: K$ such that 
$\Inv([\chi ]) = \{ (\frac{1}{3},7) \} $. 
Now take $G:=H \times D_{16} $ and $\psi := \chi \otimes \varphi $, where 
$\varphi \in \Irr(D_{16})$ is faithful of degree 2 and $\Q(\varphi ) = \Q[\sqrt{2}]$. 
Then $[\psi ]$ is the class of a Schur algebra as in Corollary \ref{odddegD}. 


\section{The solvable degree of a division algebra} 

\subsection{The main theorem and some examples} 

\begin{theorem}\label{main}
	Let $G$ be a finite solvable group and $\chi \in \Irr (G) $.
Put    $K:=\Q(\chi )$ and 
	recall that 
	$$P(\chi ) = \prod \{ p\in \PP \mid m_p(\chi ) > 1 \} .$$
	Then the following two divisibility conditions hold: 
	\begin{itemize}
		\item[(a)] $[F(K[\zeta _{P(\chi)} ]) : K] $ divides $\chi (1)$. 
		\item[(b)] $[K[\zeta _{2P(\chi )} ] : K] $ divides $\chi (1)$.
	\end{itemize}
       \end{theorem}

       \begin{remark}
	       Clearly (a) implies (b) if either $2\not\in {\mathcal P}(\chi )$ 
	       or the prime $2$ is ramified in $K/\Q $. 
       \end{remark}
  
 \begin{remark}
 The authors are aware of two papers, \cite{Pendergrass} and \cite{YI},
 that deal with  variations of the degree of a Schur algebra.
 Pendergrass shows that for any finite group $G$ and any $\chi \in \Irr(G)$ 
	 we have 
	 $\chi(1) \geq [K[\zeta _{2P(\chi )} ] : K] $ where $K:=\Q(\chi )$.
	 The paper \cite{YI} claims that Theorem \ref{main} (b) 
	 holds for arbitrary groups $G$, which is known to be wrong 
	 (see for instance Remark \ref{degeq1}).
 \end{remark}

       For some Schur algebras $D$ we can find solvable groups $G$ 
       where $\chi(1) = [K[\zeta _{2P(\chi)} ] : K]$, 
       as the well known examples in the following three corollaries show:

\begin{corollary} \label{cor1}
	Let $p$ be an odd prime and $K\leq \Q[\zeta _p]$ be such that 
	$[K:\Q ] $ is odd. Then $K$ is a real number field and 
	$t:=[\Q[\zeta _p]:K]$ is even. 
	Let $D$ be the quaternion algebra with center $K$ ramified at all
	infinite places of $K$ and at the unique place dividing $p$, 
	so ${\mathcal P}([D]) = \{ \infty, p \}$.
	Then $\deg_{\sol}(D) = t/2$.
\end{corollary}

\begin{proof}
	Let $\sigma $ be a generator of $\Gal (\Q[\zeta _p]/K )$.
	Then $\sigma (\zeta _p) = \zeta _p ^a $ for some $1\leq a \leq p-1$. 
	Put 
	$$G:= \langle x, y , z \mid x^{p} = 1, x^y = x^a, y^t=z, z^2 =1, z \mbox{ central } \rangle \cong C_{2p} . C_t .$$ 
	Up to the action of the outer automorphism group there is a 
	unique faithful $\chi \in \Irr(G)$. Then 
	$\chi (1) = t$ and $[\chi ] = [D ]$.
\end{proof}

\begin{corollary} \label{degsolquat}
The solvable degree of a rational quaternion division algebra $D$ is
	$$\deg_{\sol}(D) = \frac{[\Q[\zeta _{2P(D)}]:\Q]}{2}.$$
\end{corollary} 

\begin{proof} 
	For the odd primes $p$ dividing $P(D)$ we take $K=\Q $ in Corollary 
	\ref{cor1}, define $G_p:= C_{2p}.C_{p-1}$ 
	and $\chi _p\in \Irr(G_p)$ 
	to be the group and the irreducible character   from
	the proof of Corollary \ref{cor1}. 
	If 2 ramifies in $D$ 
then we put $G_2:= Q_8 := C_4.C_2$ to be the 
	quaternion group of order $8$ and $\chi_2 \in \Irr(G_2)$ with
	$\chi_2(1) =2$.
	Let $G$ denote the central product of all $G_p$ with $m_p(D) =2$
	and $\chi := \prod _{p \mid P(D)} \chi_p \in \Irr(G)$. 
	Then $\chi (1) = [\Q[\zeta _{2P(D)}]:\Q] $ and $[\chi ] = [D] $.
\end{proof}

Examples with $m(\chi ) > 2$ and $\deg_{\sol}([\chi ])=1$ 
can be constructed from \cite{Amitsur}. 

\begin{corollary} \label{cor2} (see also \cite[Section 2]{Unger}, 
	\cite[Theorem 8.2]{Turull})
	Let $p$ be an odd prime, $p-1 = m n $ with $m>1$ and $\gcd(m,n)=1$.
	Let $E\leq \Q[\zeta _p]$ be the subfield of index $m = [\Q[\zeta _p]:E ]$ and 
	put $K=E[\zeta _{m}]$. 
	Let $B:=\{ 1\leq b \leq m-1 \mid \gcd(m,b) =1 \} \cong (\Z/m\Z)^{\times } .$ 
	For any $b\in B$  let $D$ be the 
	Schur algebra with center $K$ and $\Inv (D) = \{ (\frac{b}{m},p ) \} .$
	Then $D$ has index $m$ and $\deg_{\sol}(D) = 1$. 
\end{corollary}

\begin{proof}
	Let $\sigma $ be a generator of $\Gal (\Q[\zeta _p]/E )$.
        Then $\sigma (\zeta _p) = \zeta _p ^a $ for some $1\leq a \leq p-1$
	such that $a+p\Z \in (\Z/p\Z )^{\times }$ is of order $m$. Put
	$$G:= \langle x, y , z \mid x^{p} = 1, x^y = x^a, y^{m}=z, z^{m} =1, z \mbox{ central } \rangle \cong C_{p} : C_{m^2} .$$ 
	Then the center of $G$ is of order $m$ and generated by $z$. 
	Let $\chi _1\in \Irr(G)$ be an irreducible character (of degree $m$)
	with $\chi _1(z) = m \zeta _{m}$.
	The Galois group $\Gamma := \Gal(\Q[\zeta _{m}]/\Q )$ acts regularly 
	on the set of prime
	ideals of $K$ that divide $p$ and also on the set 
	$$X:=\{ \chi_1^{\gamma }  \mid \gamma \in \Gamma  \} = \{ \chi_i \mid i\in B \} $$
	with $\chi _i (z) = m\zeta _{m}^{i} $.
Using the formula for computing the Hasse invariants of these characters 
in \cite[Theorem 8.2]{Turull} we find that 
	$$\{ \Inv([\chi _i]) \mid i \in B \}  =
	\{ (\frac{b}{m} , p) \mid b\in B \}.$$ 
	As the index, $m$, of $D$ is the character degree $\chi _1(1)$, we get 
	$\deg_{\sol}(D) =1$.
\end{proof}


\subsection{Proof of Theorem \ref{main}} 

This section is devoted to a proof of Theorem \ref{main}. 
The proof is split into four lemmas, where we always keep the 
notation of Theorem \ref{main} and the 
following 
\begin{assumption}\label{assumption}
$G$ is a finite solvable group
and $\chi \in \Irr (G)$ is a faithful irreducible character of $G$ 
that is a counterexample to Theorem \ref{main}. 
We also assume that $G$ is minimal, i.e. 
Theorem \ref{main} holds for all proper subgroups $U <  G$ and all 
$\psi \in \Irr(U)$.
\end{assumption}

\begin{lemma} \label{TisNormal} 
	Let $N\unlhd G$ be a normal subgroup of $G$. 
	Then $\chi _{|N } = e(\psi_1+\ldots + \psi _t)$ with $e\in \N$ 
	and $\psi _i \in \Irr(N)$  for $i=1,\ldots , t$.
	Put $\psi := \psi _1$ 
 and $T:= I_G(\psi )$ the inertia group of $\psi $. 
	Then $T\unlhd G$ is normal in $G$.
	\\
	Let $\alpha \in \Irr(T) $  be the Clifford correspondent of 
	$\chi $ lying above $\psi $, i.e. $\alpha _{|N} = e \psi $  and 
	$\chi = \alpha ^G$, and put $L:=\Q(\alpha )$. 
	Then $[L:K] = [G:T]$ and $G/T \cong \Gal (L/K) $.
\end{lemma}

\begin{proof}
Put $H:=N_G(T)$ and let $\beta := \alpha^H$. 
 Then $T \leq H \leq G$ and  $ \beta^G = (\alpha^H)^G = \alpha^G = \chi$. 
In particular $\beta \in \Irr(H)$ and $K \leq \Q(\beta ) \leq L$.

We first show that $\Q(\beta) = K$.
Let $\sigma \in \Gal(L/K)$. Since $\chi^\sigma = \chi$, 
	the character $\psi^\sigma \in \Irr(N)$ 
	is again an irreducible constituent of $\chi_{\mid N}$, so 
	there is $x\in G $ with 
	 $\psi^\sigma = \psi^x$. 
      On the other hand,
		   $T^x = I_G(\psi ^x) =  I_G(\psi^\sigma) = I_G(\psi) = T $, 
                        since the Galois conjugation commutes with the conjugation 
			action. 
	So $x \in N_G(T) = H$ and hence $\beta ^{\sigma } = \beta $. 

	Now $\Q(\beta ) = \Q(\chi) = K$ so  $[\beta ] = [\chi ] \in \Br(K)$ 
	by Corollary \ref{brauer1}. 
	The minimality of $G$ hence implies that $H=G$ and hence $T\unlhd G$.

	As in the proof of Theorem \ref{Yamadagen} (a) we now put
	$$H:=\{ g\in G \mid \exists \sigma \in \Gal(L/K ) : \alpha (gng^{-1}) = \sigma(\alpha (n) ) \mbox{ for all } n\in T \} .$$
	Then $\beta := \alpha ^H \in \Irr(H) $ satisfies $\Q(\beta ) = K $ 
	and  $\chi = \beta ^G$. This again implies that $[\beta ] = [\chi ]$ 
	 by Corollary \ref{brauer1}.
        The minimality of $G$ implies that $H=G$ and hence 
	$G/T = H/T \cong \Gal (L/K)$.
\end{proof}


We now construct a chief series of $G$ as follows: 
Let $G_1$ be a minimal normal subgroup of $G$ and 
$\chi _1 \in \Irr(G_1)$ be a constituent of $\chi _{|G_1} $. 
Let $T_1:= I_G(\chi _1)$ be the inertia group of $\chi _1$. 
Then $$G_1 \leq T_1 \leq G$$ and also $T_1$ is normal in $G$. 
If $T_1\neq G_1$ then 
let $G_2\leq T_1$ be such that $G_2/G_1$ is a minimal normal subgroup of $G/G_1$ 
and put $T_2 := I_G(\chi _2)$ where $\chi_2 \in \Irr (G_2)$ is a constituent of $\chi _{|G_2} $ such that $(\chi_2)_{|G_1} $ is a multiple of $ \chi _1$.

As $I_G(\chi_2)$ also stabilises the restriction of $\chi _2$ to the normal 
subgroup $G_1$ we have that $T_2 \leq T_1$. 
We continue like this until $G_n = T_n$ to obtain a normal series 
$$ 1=G_0 < G_1 < G_2 < \ldots < G_n  =T_{n}  \unlhd G $$ 
such that $Q_j:=G_j/G_{j-1} $ are minimal normal subgroups of $G/G_{j-1}$ for 
$j=1,\ldots , n$. 
In particular all these $Q_j$ are elementary abelian. 

\begin{lemma} \label{oddzeta}
	Put $T:=T_n$ and write
	$\chi _{|T} = \alpha _1+ \ldots + \alpha _t$. Put $\alpha := \alpha_1$. 
	Then 
		 $\zeta _{P(\chi )} \in \Q(\alpha )$.
\end{lemma}

\begin{proof}
For $j=1,\ldots , n$ the restriction of $\chi _j$ to $G_{j-1}$ is 
	a multiple of $\chi _{j-1}$, in particular 
	$$ \Q(\chi_1) \leq \Q (\chi_2 ) \leq \ldots \leq \Q(\chi_n)\leq \Q(\alpha ) .$$
	\\
	Assume that $p\in {\mathcal P}(\chi )$ is odd 
	and that $\Q[\zeta _p] \not\subseteq \Q(\alpha )$. 
	Then $\Q[\zeta _p] \not\subseteq \Q(\chi _j )$ for all $j$.
We then show inductively that every $\chi_j$ is $p'$-special: 
\\
	Clearly $\chi_0 = 1$ is $p'$-special. \\
	Let $j>0$ and assume that $\chi _{j-1}$ is $p'$-special. 
	\\
	If $ Q_j$ is not a $p$-group, then by 
	Lemma \ref{pspeciallemma} also $\chi_j$ is $p'$-special.  \\
	Assume that  $Q_j $ is an elementary abelian $p$-group. 
	As $G_j\leq I_G(\chi_{j-1})$ 
	 Lemma \ref{pspeciallemma} says that the character $\chi _{j-1}$ has a 
	unique $p'$-special extension $X:=\hat{\chi}_{j-1} \in \Irr(G_j)$. 
	Then $X_{|G_{j-1}} = \chi _{j-1} \in \Irr(G_{j-1})$. 
	By Theorem \ref{Gallagher} the induced character is
	$$ \chi_{j-1} ^{G_j} = \sum _{\lambda \in \Irr(Q_j) } \lambda X .$$ 
	By Frobenius reciprocity also $\chi _j$ is a constituent of 
	$\chi _{j-1}^{G_j}$, so there is some $\lambda \in \Irr(Q_j)$ 
	such that $\chi_j = \lambda X$. As all irreducible characters
	$\lambda $ of the abelian group $Q_j$ are linear, 
	$\chi _j(1) = X(1) = \chi _{j-1}(1) $. 
	Now $\chi _{j-1}$ is $p'$-special, hence $\chi _j(1)$ is not a multiple
	of $p$. By assumption $\Q[\zeta _p] $ is not a subfield of $\Q(\chi _j)$, 
	then also the determinantal order of $\chi _j$ is not divisible 
	by $p$ and hence $\chi _j$ is $p'$-special.
	\\
	In particular also $\alpha $ is $p'$-special. 
	Repeting the argument with $\alpha = \alpha _i$ instead of $\alpha = \alpha _1$ 
	(for some $1\leq i \leq t$)
	we see that all the $\alpha _i$ are $p'$-special. 
	By Lemma \ref{SchurIndex1} 
	restriction $\alpha _i'$ of $\alpha_i $ to the $p$-regular elements of
	$T$ is an irreducible Brauer character.
\\
	Now \cite[Proposition 1.6]{GAJENDRAGADKAR1979237}
	(see also \cite[Lemma 3.5]{Isaacspaper})
	shows that the restriction of  
	the $\alpha _i$ to a Hall $p'$-subgroup of $T$ remain pairwise distinct, 
	so also the Brauer characters $\alpha _i'$ 
	are $t$ distinct irreducible $p$-Brauer characters. 
	So the restriction $\chi '$ of $\chi $ to the 
	$p$-regular elements of $G$ is the induced character
	$$\chi ' = ( \alpha ') ^G$$ 
	and hence an irreducible Brauer character of $G$.
	By Remark \ref{remSchurIndex1} this implies that 
	$m_p(\chi ) =1 $ contradicting the assumption.
\end{proof} 

For $G$, $\chi $, 
$K=\Q(\chi )$ as in Assumption \ref{assumption}
and $T$, $\alpha \in \Irr(T)$, $E:=\Q(\alpha )$ and $t=[G:T]=[E:K]$ 
as in Lemma \ref{oddzeta} 
we hence have
	\begin{itemize}
		\item[(a)] $K\leq E $ and $[E : K] = t$ (by Lemma \ref{TisNormal}). 
		\item[(b)] $\zeta _{P(\chi )} \in E$ 
			(by Lemma \ref{oddzeta}).
		\item[(c)] $[F(E) : E ] \bigm| \alpha (1)$
			(by Theorem \ref{Cram}).
		\item[(d)] $\zeta _4 \in F(E) $ if $2$ ramifies in 
			$E$ over $\Q $. 
	\end{itemize} 

\begin{corollary} \label{corzetaodd} 
	$[F(K[\zeta _{P(\chi )} ]) : K ] \bigm| \chi (1) $, so $\chi $ 
	satisfies Theorem \ref{main} (a). 
\end{corollary} 

\begin{proof}
	As $K[\zeta _{P(\chi )} ]  \leq E$ we have that 
	$$[F(K[\zeta _{P(\chi )} ]) : K ] \bigm| [F(E) : K] = [F(E):E][E:K]  \bigm| \alpha(1) t = \chi(1) .$$
\end{proof}

In particular $\chi $ satisfies Theorem \ref{main} if 
$2\not\in {\mathcal P}(\chi )$. 
So in the following we assume that $m_2(\chi ) \neq 1$ and hence 
 $m_2(\chi ) =2$.

\begin{lemma} \label{remainsodd} 
If $2$ is ramified in $E/\Q$, then 
$[F(K[\zeta _{2 P(\chi )}]) : K ] \bigm| \chi (1) $.
In particular $\chi $ satisfies Theorem \ref{main} (b).
\end{lemma}

\begin{proof} 
	If $2$ ramifies in $E$ then it also ramifies in 
	$F(E)$, so $\zeta _4 \in F(E)$. In particular 
	$K[\zeta _{2P(\chi )}]$ is a subfield of $F(E)$ and hence 
	also $F(K[\zeta _{2P(\chi )}]) \leq F(E) $ so again
	$$[F(K[\zeta _{2P(\chi )}]) : K] \bigm| [F(E) : K] = 
	[F(E) : E] [E: K] \bigm| \alpha(1) t = \chi (1) .$$
\end{proof} 

\begin{lemma}\label{evenalpha} 
	If $\alpha (1) $ is even then 
	$[K[\zeta _{2 P(\chi )}] : K ] \bigm| \chi (1) $,
so $\chi $ satisfies Theorem \ref{main} (b).
\end{lemma}

\begin{proof}
	We have that $K[\zeta _{P(\chi )}] \leq E$
	and $[K[\zeta _{2 P(\chi )}] : K[\zeta _{P(\chi )}] \bigm| 2$. 
	So 
$$[K[\zeta _{2 P(\chi )}] : K ] \bigm| 2 [E:K] = 2 t \bigm| \alpha(1) t = \chi(1) .$$
\end{proof}

By Lemma \ref{remainsodd} and \ref{evenalpha} we are left with 
the situation that $\Q(\alpha )/\Q $ is not ramified at the prime 2 and
$\alpha (1)$ is odd.  Then $m_2(\alpha ) =1$ (see Remark \ref{mdivdeg} and 
Theorem \ref{unidist} (iii)).
Therefore 
Theorem \ref{Yamadagen} tells us that 
$m_2(\chi ) = 1$.
This finishes the proof of Theorem \ref{main}. 

{\sc AI Declaration.} 
Some proofs and statements from the first version have been corrected based on 
a very helpful AI-assisted preliminary report provided by the editor.

\end{document}